\documentclass[11pt,a4paper]{article}
\usepackage{amsmath, amsfonts, amsthm, amssymb, graphicx, color, dsfont, enumerate, hyperref}
\usepackage[numbers,square, comma, sort&compress]{natbib}

\usepackage{enumitem}

\theoremstyle{plain}
\newtheorem{satz}{Satz}[]
\newtheorem{theorem}[satz]{Theorem}
\newtheorem{lemma}[]{Lemma}

\newtheorem{problem}[]{Problem}
\newtheorem{corollary}[]{Corollary}
\theoremstyle{remark}

\theoremstyle{definition}

\begin{document}

\title
{Counting cycles in planar triangulations}
\author{
{\sc On-Hei Solomon LO}\footnote{D{\'e}partement d'informatique et de
	recherche op{\'e}rationnelle, Universit{\'e} de
	Montr{\'e}al, Pavillon Andr{\'e}-Aisenstadt, 2920, chemin de la Tour, Montr{\'e}al, Qu{\'e}bec
	H3T 1J4, Canada. E-mail address: ohsolomon.lo@gmail.com} \ and {\sc Carol T. ZAMFIRESCU}\footnote{Department of Applied Mathematics, Computer Science and Statistics, Ghent University, Krijgslaan 281 - S9, 9000 Ghent, Belgium and Department of Mathematics, Babe\c{s}-Bolyai University, Cluj-Napoca, Roumania. E-mail address: czamfirescu@gmail.com}}
\date{}

\maketitle
\begin{center}
\vspace{2mm}
\begin{minipage}{125mm}
{\bf Abstract.} We investigate the minimum number of cycles of specified lengths in planar $n$-vertex triangulations $G$. It is proven that this number is $\Omega(n)$ for any cycle length at most $3 + \max \{ {\rm rad}(G^*), \lceil (\frac{n-3}{2})^{\log_32} \rceil \}$, where ${\rm rad}(G^*)$ denotes the radius of the triangulation's dual, which is at least logarithmic but can be linear in the order of the triangulation. We also show that there exist planar hamiltonian $n$-vertex triangulations containing $O(n)$ many $k$-cycles for any $k \in \{ \lceil n - \sqrt[5]{n} \rceil, \ldots, n \}$. Furthermore, we prove that planar 4-connected $n$-vertex triangulations contain $\Omega(n)$ many $k$-cycles for every $k \in \{ 3, \ldots, n \}$, and that, under certain additional conditions, they contain $\Omega(n^2)$ $k$-cycles for many values of $k$, including~$n$.

\smallskip


\smallskip


\end{minipage}
\end{center}

\bigskip

\bigskip

\section{Introduction}

In this paper, a \emph{triangulation} shall be a plane graph unless explicitly stated otherwise, i.e.\ an embedded planar graph, in which every face is a triangle, but we exclude $K_3$ so that a triangulation shall here always be 3-connected. In a graph $G$, the length of a shortest cycle in $G$ is called its \emph{girth}, and the length of a longest cycle its \emph{circumference}. The latter will be abbreviated to ${\rm circ}(G)$. For a graph $G$, its \emph{cycle spectrum} is the set of all lengths of cycles occurring in $G$. We will here be interested in counting cycles in triangulations. For a graph $G$, a set $X \subset V(G)$ is a \emph{cut} if $G - X$ has more connected components than $G$. A cut on $k$ vertices is a \emph{$k$-cut}.

Hitherto, most results on the enumeration of cycles in $n$-vertex triangulations $G$ focus on the two ends of the cycle spectrum. Historically, the first result which can be interpreted as a statement on counting cycles in triangulations is Euler's formula, which implies that $G$ contains $\Omega(n)$ many 3-cycles. Hakimi and Schmeichel~\cite{HS79} presented tight upper and lower bounds on the number of 3-cycles and 4-cycles as well as those graphs for which the bounds are attained except for a few cases which were later completed by Alameddine~\cite{Al80}. In particular, it follows from~\cite{HS79} that any $n$-vertex triangulation $G$ contains $\Theta(n)$ many 3-cycles as well as $\Omega(n)$ and $O(n^2)$ many 4-cycles.
There exist infinitely many triangulations with a linear number of 3-cycles (4-cycles), for instance 4-connected (5-connected) triangulations: a 3-cycle (4-cycle) in such a graph cannot be separating, so it must be the boundary of a triangular face (the symmetric difference of the boundaries of two triangular faces sharing an edge).

Hakimi and Schmeichel also showed that there must occur $\Omega(n)$ and $O(n^2)$ many 5-cycles. They conjectured a precise upper bound of $2n^2-10n+ 12$ and this was confirmed only very recently~\cite{GPSTZ}. Moreover, they proved that in an $n$-vertex triangulation there are at least $6n$ 5-cycles, and there are infinitely many triangulations attaining this bound. Since in triangulations every inclusion-minimal cut induces a cycle, these results immediately give information on the number of 3-, 4-, and 5-cuts in triangulations. For $k > 5$, Hakimi and Schmeichel write that they ``\emph{have no interesting lower bound}'' for the number of $k$-cycles. This motivates the following question.

\begin{problem} \label{pro:problem1}
	What can be said, asymptotically, about the minimum number of $k$-cycles occurring in a given triangulation for $k > 5$?
\end{problem}

Hakimi and Schmeichel~\cite{HS79} also proved that every triangulation is \emph{weakly pancyclic}, i.e.\ it contains cycles of all lengths between its girth, which is always three, and circumference. It is however \emph{not} true that every $n$-vertex triangulation $G$ contains $\Omega(n)$ many $k$-cycles for every integer $k$ with $3 \le k \le {\rm circ}(G)$: Hakimi, Schmeichel, and Thomassen~\cite{HST79} constructed an infinite family of triangulations with exactly four hamiltonian cycles. And one cannot do better: Kratochvil and Zeps~\cite{KZ88} proved that if a triangulation different from $K_4$ contains a hamiltonian cycle, then it contains at least four of them. They also show that this theorem holds for triangulations of arbitrary surfaces.

On the one hand, the results we just mentioned illustrate that a hamiltonian triangulation may have very few hamiltonian cycles; on the other hand, a triangulation may be far from hamiltonian. In the sixties Moon and Moser~\cite{MM63} proved that there are infinitely many $n$-vertex triangulations that have circumference at most $9n^{\log_3 2}$, and 40 years later Chen and Yu~\cite{CY02} showed that this upper bound is best possible up to the constant factor.

In 1979 Hakimi, Schmeichel, and Thomassen proved that every 4-connected $n$-vertex triangulation contains at least $n/\log_2 n$ hamiltonian cycles. This was improved in 2018 when Brinkmann, Souffriau, and Van Cleemput~\cite{BSV18} showed that there is a linear number of hamiltonian cycles in 4-connected triangulations. For very recent developments concerning the enumeration of hamiltonian cycles in triangulations, we refer the reader to~\cite{AAT20, Lo20, LY21, LQ}. It follows from a classic paper of Tutte~\cite{Tu56} that in 4-connected $n$-vertex triangulations the number of $(n - 1)$-cycles is also at least linear in $n$, and one can infer from papers of Thomas and Yu~\cite{TY94} and Sanders~\cite{Sa96} that the same holds for $(n - 2)$- and $(n - 3)$-cycles. (In fact, \cite{TY94} even yields a quadratic number of $(n - 2)$-cycles; we shall come back to this deep result later on.) This naturally leads to the following question.

\begin{problem} \label{pro:problem2}
	Are $4$-connected $n$-vertex triangulations linearly pancyclic, i.e.\ do they contain $\Omega(n)$ many $k$-cycles for every $k \in \{ 3, \ldots, n \}$?
\end{problem}

Mohar and Shantanam~\cite{MS} have recently proven that in a 4-connected triangulation on $n$ vertices, every edge is contained in cycles of $n - 2$ pairwise distinct lengths. They also showed that there exists $\frac{5}{12} \le \lambda \le \frac23$ such that any edge in any planar 4-connected $n$-vertex graph is contained in cycles of $\lambda(n-2)$ pairwise distinct lengths, and conjectured that $\lambda = \frac23$. The lower bound $\frac{5}{12}(n - 2)$ was improved to $\frac{n}{2} + 1$ in~\cite{Lo22}. 

We mention that many problems in extremal graph theory revolve around finding, for a fixed graph $H$, the \emph{maximum} number of subgraphs isomorphic to $H$ in an $n$-vertex graph in some class ${\cal G}$. For graphs $H$ and $G$ and a surface $\Sigma$, let $C(H,\Sigma,n)$ be the maximum number of copies of $H$ in $G$, where the maximum is taken over all $n$-vertex graphs $G$ that embed in $\Sigma$. A very recent manuscript of Huynh, Joret, and Wood~\cite{HJW} gives a panoramic view of results in this area and determines, for any fixed surface $\Sigma$ and any fixed graph $H$, the asymptotic behaviour of $C(H,\Sigma,n)$ as $n \to \infty$.

Let $G$ be an embedded graph. We denote by $G^*$ the dual of $G$. If $G$ is a plane graph, the \emph{weak dual} of $G$ is obtained from $G^*$ by deleting the vertex corresponding to the unbounded face of $G$. For a vertex $v$ in $G$ let $v^*$ be the face of $G^*$ corresponding to the vertex $v$. For distinct vertices $v$ and $w$ in a graph $G$, we call a path with end-vertex $v$ a \emph{$v$-path}, and a $v$-path with end-vertex $w$ a \emph{$vw$-path}. For a path $P$, its \emph{length} is $|E(P)|$. The \emph{distance} $d_G(v,w)$ between $v$ and $w$ is defined as the length of a shortest $vw$-path in $G$. Whenever $G$ is clear from context, we replace $d_G$ by $d$. For a graph $G$, its \emph{radius} and \emph{diameter} are defined as $${\rm rad}(G) := \min_{v \in V(G)} \max_{w \in V(G)} d(v,w) \qquad {\rm and} \qquad {\rm diam}(G) := \max_{v \in V(G)} \max_{w \in V(G)} d(v,w),$$
respectively. In this paper, a face will always include the facial walk bounding it. For a plane graph $G$, we denote by $F(G)$ its set of faces. For a cycle $C$ in $G$, let ${\rm int}C$ be the set of all vertices inside of $C$ but not on $C$, and put ${\rm Int}C := V(C) \cup {\rm int}C$. The former will be called the \emph{interior} of $C$, the latter the \emph{closed interior} of $C$. The \emph{exterior} ${\rm ext}C$ and \emph{closed exterior} ${\rm Ext}C$ are defined analogously. For any $v \in V(G)$, the \emph{neighbourhood} $N_G(v)$ of $v$ in $G$ is defined to be the set of vertices adjacent to $v$ in $G$. We simply write $N(v)$ if it causes no ambiguity.

In Section~2 we show that the minimum number of $k$-cycles occurring in an $n$-vertex triangulation $G$ is $\Omega(n)$ for any $k \in \{ 3, \ldots, 3 + \max \{ {\rm rad}(G^*), \lceil (\frac{n-3}{2})^{\log_32} \rceil \} \}$. In Section~3, motivated by finding hamiltonian triangulations---a superset of 4-connected triangulations---with few long cycles, it is shown that there exist hamiltonian $n$-vertex triangulations containing $O(n)$ many $k$-cycles for any $k \in \{ \lceil n - \sqrt[5]{n} \rceil, \ldots, n \}$. Moreover, we prove that 4-connected $n$-vertex triangulations contain $\Omega(n)$ many $k$-cycles for any $k \in \{ 3, \ldots, n \}$, and that, under certain additional conditions, they contain $\Omega(n^2)$ $k$-cycles for many values of $k$, including $n$. The article concludes with comments on directions of future research and open problems in Section~4.

\section{Linearly many short cycles}

We here show that an $n$-vertex triangulation contains linearly many cycles of every length that is at most the radius of the triangulation's dual or $O(n^{\log_32})$, whichever is greater.

The following definition will be useful. Let $G$ be a 2-connected plane graph. For a connected subgraph $H$ of $G^*$, put $$H^* = \bigcup \{ v^* : v \in V(H) \},$$ i.e.\ a union of faces. Furthermore, let $\partial H^*$ be the subgraph of $H^*$ defined as
$$(V(H^*), \{ vw \in E(H^*) : vw \text{\;is incident with exactly one face from\;} H^* \})$$
from which all isolated vertices are removed. 

We will make use of the following lemma. A \emph{near triangulation} is a 2-connected plane graph all of whose bounded faces are triangular.

\begin{lemma} \label{lem:path-near-triangulation}
	Let $G$ be a near triangulation on at least four vertices and $D$ be the weak dual of $G$. If at most two bounded triangular faces in $G$ contain two edges of $\partial D^*$ and any bounded triangular face in $G$ that has some edge of $\partial D^*$ has all three of its vertices contained in $\partial D^*$, then $D$ is a path.
\end{lemma}

\begin{proof}
	We prove the statement by induction on the number of vertices of $G$. If $|V(G)| = 4$, then $G$ has exactly two bounded faces and $D$ is a path of length one.
	
	If $|V(G)| > 4$, then there must be a triangular face $\Delta$ on vertices $v_1, v_2, v_3$ such that $\Delta$ has only one edge $v_1 v_2$ contained in $\partial D^*$ and $v_3$ is in $\partial D^*$ as well. Note that $D - \Delta^*$ has precisely two components $H_1, H_2$. Denote by $G_1$ and $G_2$ the near triangulations induced by $D - H_1$ and $D - H_2$, respectively. We may assume that $G_1$ has at most two bounded triangular faces each containing two edges of $\partial D_1^*$, where $D_1$ is the weak dual of $G_1$. By the induction hypothesis, $D_1$ is a path. Moreover, this implies that $G_2$ has at most two bounded triangular faces each containing two edges of $\partial D_2^*$ as $D_1$ has an end-vertex other than $\Delta^*$,  where $D_2$ denotes the weak dual of $G_2$. Again, by the induction hypothesis, $D_2$ and hence $D$ are paths.
\end{proof}

For our first theorem's proof we also need the following technical framework. Brinkmann, Souffriau, and Van Cleemput~\cite{BSV18} introduced so-called {counting bases} in order to enumerate hamiltonian cycles in planar graphs. Making use of their method, we are able to solve Problem~\ref{pro:problem2} affirmatively; the details will be given in the next section. We recall the main ingredients of this technique which we will then slightly modify so that we can count cycles of a specified length.

Let $G$ be a graph. A \emph{counting base} for $G$ is defined to be a triple $(\mathcal{P}, \{\mathcal{C}_P\}_{P \in \mathcal{P}}, \sigma)$, where $\mathcal{P} \subseteq 2^{E(G)}$, $\mathcal{C}_P$ is a family of cycles in $G$ for each $P \in \mathcal{P}$, and $\sigma: \mathcal{P} \rightarrow \mathcal{P}$ is a function satisfying the following properties.

\begin{enumerate}[label=(\roman*)]
	\item For every $P \in \mathcal{P}$, $\mathcal{C}_P$ is non-empty, and for every $C \in \mathcal{C}_P$, $P \subseteq E(C)$;
	\item for every $P \in \mathcal{P}$, $\sigma(\sigma(P)) = P \not\subseteq \sigma(P)$ and for every $C \in \mathcal{C}_P$, $\sigma(C, P) := (C - P) \cup \sigma(P) \in \mathcal{C}_{\sigma(P)}$ (here we extend the function $\sigma$ with abuse of notation) and $\sigma(\sigma(C, P), \sigma(P)) = C$;
	\item for any distinct $P_1, P_2 \in \mathcal{P}$ and any $C \in \mathcal{C}_{P_1} \cap \mathcal{C}_{P_2}$, $\sigma(C, P_1) \neq \sigma(C, P_2)$;
	\item for any $P_1, P_2 \in \mathcal{P}$ with $P_1 \cap P_2 = \emptyset$ and any $C \in \mathcal{C}_{P_1} \cap \mathcal{C}_{P_2}$, $\sigma(C, P_1) \in \mathcal{C}_{P_2}$.
	\end{enumerate}
	
Condition~(ii) immediately implies that for every $P \in \mathcal{P}$, $P$ is non-empty and for every $C \in \mathcal{C}_P$, $P \not\subseteq E(\sigma(C, P))$ and hence $\sigma(C, P) \notin \mathcal{C}_P$.
	
We will give a lower bound on the number of cycles in $\bigcup_{P \in \mathcal{P}} \mathcal{C}_P$ by relating $|\bigcup_{P \in \mathcal{P}} \mathcal{C}_P|$ to $|\mathcal{P}|$. To this end, we need to factor out the ``overlap'' when counting the number of cycles by considering the following notion. Given a counting base $\mathfrak{C} = (\mathcal{P}, \{\mathcal{C}_P\}_{P \in \mathcal{P}}, \sigma)$, we define for $P \in \mathcal{P}$ and $C \in \mathcal{C}_P$ the \emph{overlap} $$\mathbf{o}_\mathfrak{C} (C, P) := |\{P' \in \mathcal{P} : P \cap P' \neq \emptyset\textrm{ and }C \in \mathcal{C}_{P'}\}|.$$ Furthermore, put
$$\mathbf{O}_\mathfrak{C} := \max_{P \in \mathcal{P},\, C \in \mathcal{C}_P} \mathbf{o}_\mathfrak{C} (C, P).$$
Note that $\mathbf{o}_\mathfrak{C} (C, P) \ge 1$ for any $P \in \mathcal{P}$ and any $C \in \mathcal{C}_P$; in particular, $\mathbf{O}_\mathfrak{C} \ge 1$.
	
The following powerful tool is due to Brinkmann, Souffriau, and Van Cleemput~\cite{BSV18}. We include a proof for completeness' sake.
	
\bigskip
	
\noindent \textbf{Counting Base Lemma} ({\cite[Theorem~1]{BSV18}}).
\emph{Let $G$ be a graph and $\mathfrak{C} = (\mathcal{P}, \{\mathcal{C}_P\}_{P \in \mathcal{P}}, \sigma) \ne \emptyset$ a counting base for $G$. Then $$\left|\bigcup_{P \in \mathcal{P}} \mathcal{C}_P\right| \ge \frac{|\mathcal{P}|}{\mathbf{O}_\mathfrak{C}}.$$}
		
\begin{proof}
For a cycle $C$ in $G$, let $m(C) := | \{ P \in \mathcal{P} : C \in \mathcal{C}_P \}|$. We have
$$\left|\bigcup_{p \in \mathcal{P}} \mathcal{C}_P\right| = \sum_{P \in \mathcal{P}} \sum_{C \in \mathcal{C}_P} \frac{1}{m(C)}.$$
It is left to show that
$$\sum_{C \in \mathcal{C}_P} \frac{1}{m(C)} \ge \frac{1}{\mathbf{O}_\mathfrak{C}}$$
for every $P \in \mathcal{P}$, an inequality we abbreviate by $(\dagger)$.
			
Let $P$ be any element of $\mathcal{P}$, and let $C^0 \in \mathcal{C}_P$. We may assume that $m(C^0) > \mathbf{O}_\mathfrak{C}$. Then there exist $t \ge m(C^0) - \mathbf{O}_\mathfrak{C}$ and $P^1, \dots, P^t \in \mathcal{P}$ such that $C^0 \in \mathcal{C}_{P^i}$ and $P^i \cap P = \emptyset$ for every $i \in \{ 1, \dots, t \}$. Note that $C^0, \sigma(C^0, P^1), \dots, \sigma(C^0, P^t)$ are $t + 1$ pairwise distinct cycles in $\mathcal{C}_P$.
		
We claim that $m(\sigma(C^0, P^i)) \le m(C^0) - 1 + \mathbf{O}_\mathfrak{C}$. By definition,
$$m(\sigma(C^0, P^i)) = |\{ P' \in \mathcal{P} : \sigma(C^0, P^i) \in \mathcal{C}_{P'} \} |.$$
We distinguish between two cases, depending on whether or not $C^0 \in \mathcal{C}_{P'}$ holds. As $\sigma(C^0, P^i) \notin \mathcal{C}_{P^i}$, there are at most $m(C^0) - 1$ many $P'$ satisfying $\sigma(C^0, P^i) \in \mathcal{C}_{P'}$ and $C^0 \in \mathcal{C}_{P'}$. For $P'$ satisfying $\sigma(C^0, P^i) \in \mathcal{C}_{P'}$ and $C^0 \notin \mathcal{C}_{P'}$, by the definition of a counting base, we have that $P' \cap \sigma(P^i) \neq \emptyset$. Therefore, there are at most $\mathbf{o}_\mathfrak{C}(\sigma(C^0, P^i), \sigma(P^i))$ many such $P'$. This justifies the claim.
		
As $\mathbf{O}_\mathfrak{C} \ge 1$, we have
$$\sum_{C \in \mathcal{C}_P} \frac{1}{m(C)} \ge \frac{1}{m(C^0)} + \sum_{i=1}^t \frac{1}{\sigma(C^0, P^i)} \ge \frac{1}{m(C^0)} + \frac{m(C^0) - \mathbf{O}_\mathfrak{C}}{m(C^0) - 1 + \mathbf{O}_\mathfrak{C}} \ge \frac{m(C^0) - \mathbf{O}_\mathfrak{C} + 1}{m(C^0) + \mathbf{O}_\mathfrak{C} - 1}.$$ Since $\frac{m(C^0) - \mathbf{O}_\mathfrak{C} + 1}{m(C^0) + \mathbf{O}_\mathfrak{C} - 1}$ is non-decreasing in $m(C^0)$ and $m(C^0) \ge \mathbf{O}_\mathfrak{C} + 1$, $(\dagger)$ holds.
\end{proof}

We also require the following lemma (more precisely, its second part---its first part will be useful later on). Thereafter, we can state and prove our first main result.

\begin{lemma} \label{lem:weaklypancyclic}
	Let $G$ be a near triangulation, $C$ be the boundary cycle of the unbounded face of $G$, and $v_1 v_2, v_2 v_3, v_3 v_4$ be three consecutive edges in $C$. Then the following hold.
	\begin{enumerate}[label=(\roman*)]
		\item There exists a $k$-cycle in $G$ containing $v_1 v_2 v_3$ for every $k \in I := \{ \min\{ |E(C)|, |N(v_2)| + 1 \}, \ldots, \max\{ |E(C)|, |N(v_2)| + 1 \} \}$.
		\item If $v_1 \neq v_4$ and $v_2 v_3 v_4 v_2$ is a triangular facial boundary, then there exists a $k$-cycle in $G$ containing $v_1 v_2 v_3 v_4$ for every $k \in I$. 
	\end{enumerate}
\end{lemma}
\begin{proof}
	Let $D_0$ be the weak dual of $G$.
	We construct a sequence of $\ell := |V(D_0)| - |N(v_2)| + 1$ subgraphs $D_1, \dots, D_\ell$ of $D_0$ by removing vertices one-by-one from $D_0$ as follows. Suppose we have already constructed $D_i$ for some $i$ with $0 \le i < \ell$. We may take $v \in V(D_i)$ such that $D_i - v$ remains connected and the triangular face $v^*$ does not contain $v_2$ but some edge in $\partial D_i^*$. Define $D_{i + 1} := D_i - v$.
	It is not hard to see that this iterative procedure holds, every $\partial D_i^*$ is a cycle containing $v_1 v_2 v_3$, $\partial D_\ell^*$ is the cycle on $\{v_2\} \cup N(v_2)$, and
	$$\big\vert |E(\partial D_i^*)| - |E(\partial D_{i + 1}^*)| \big\vert = 1$$
	for all $i \in \{ 0, \ldots, \ell - 1 \}$. Therefore, among these $\ell + 1$ cycles $\partial D_i^*$, there are $k$-cycles containing $v_1v_2v_3$ for all $k \in I$. This proves the first statement.
	Moreover, if $v_1 \neq v_4$ and $\{v_2, v_3, v_4\}$ induces a triangular face, then every cycle containing $v_1 v_2 v_3$ must contain $v_3 v_4$ as well. This thus proves the second statement.
\end{proof}

\begin{theorem} \label{thm:linearlyshortcycles}
	Let $G$ be an $n$-vertex triangulation. For every integer $k$ with $3 \le k \le 3 + \max \{ {\rm rad}(G^*), \lceil (\frac{n-3}{2})^{\log_32} \rceil \}$ there are $\Omega(n)$ many $k$-cycles in $G$. 
\end{theorem}
\begin{proof}
	We first show that for every integer $k$ with $3 \le k \le 3 + {\rm rad}(G^*)$ there are $\Omega(n)$ many $k$-cycles in $G$. Let $v$ and $w$ be two vertices in $G^*$. We claim that there is a $v w$-path $P$ in $G^*$ such that $P$ is an induced path in $G^*$ and $\partial P^*$ is a cycle containing two edges of $v^*$ and two edges of $w^*$. Assume that the unbounded face of $G$ is neither $v^*$ nor $w^*$, and denote by $D_0$ the weak dual of $G$. We construct a sequence of subgraphs $D_1, D_2, \ldots$ of $D_0$ by successively removing, one-by-one, vertices from $D_0$ as follows. If there is a vertex $u$ in $D_i$ that is neither $v$ nor $w$ such that either $\partial D_i^*$ contains two edges of $u^*$, or $\partial D_i^*$ contains one edge of $u^*$ but not all vertices of $u^*$, then we obtain $D_{i + 1}$ by deleting $u$ from $D_i$. It is clear that $D_i$ is always a connected induced subgraph of $D_0$ containing $v, w$ and $\partial D_i^*$ is a cycle for all $i$.

Let $D_\ell$ be the graph we obtain when no further vertex can be removed. Then there is no vertex $v_0 \in V(G)$ lying in the interior of $\partial D_\ell^*$, as otherwise there would exist pairwise distinct vertices $v_1, v_2, v_3$ in $\partial D_\ell^*$ which can be reached from $v_0$ by a $v_0 v_1$-path, a $v_0 v_2$-path, and a $v_0 v_3$-path, all three of which are pairwise internally disjoint. It is clear that in each of the three regions of the closed interior of $\partial D_\ell^*$ formed by these three paths, there is a bounded triangular face which either has two edges of $\partial D_\ell^*$ or has an edge of $\partial D_\ell^*$ and a vertex not in $\partial D_\ell^*$, yielding a contradiction, as we should then remove the corresponding vertex of one such face that is neither $v$ nor $w$ from $D_\ell$. Therefore, we can apply Lemma~\ref{lem:path-near-triangulation} to deduce that $D_\ell$ is a path. It is also obvious that $v$ and $w$ are the end-vertices of $D_\ell$. This implies that $\partial D_\ell^*$ is a cycle containing two edges of $v^*$ and two edges of $w^*$.
	
	Let $v$ be any vertex in $G^*$. By the claim above, we can find a $v$-path $P$ in $G^*$ of length $r := {\rm rad}(G^*)$ such that $P$ is an induced path and $\partial P^*$ is a cycle.
	We denote the vertices of $P$ by $v, w_1, \ldots, w_r$. Let $P_i \subseteq P$ be the $v w_i$-path for all $i \in \{ 1, \ldots, r \}$. As $\partial P^*$ is a cycle, each $\partial P_i^*$ is a cycle. The set ${\cal C} := \{ \partial P_i^* \}_{i=1}^r$ contains exactly one $k$-cycle for each $k \in \{ 4, \ldots, r+3\}$, and there exist edges $e$ and $f$ in $v^*$ such that the edge-set of the intersection of any cycle in ${\cal C}$ with $v^*$ is $\{ e, f \}$. We call this property~$(\dagger)$.
	
	We repeat the above procedure for every vertex in $G^*$, of which there are $2n - 4$ by Euler's formula. In this fashion, due to $(\dagger)$, any given cycle was counted at most twice. Hence, there exist at least $n - 2$ $k$-cycles for every $k$ with $3 \le k \le r + 3$. We remark that the cycles found in the aforementioned proof have an empty interior or exterior. We have completed the first part of the proof.

\smallskip

We now prove that for every integer $k$ with $3 \le k \le 3 + \lceil (\frac{n-3}{2})^{\log_32} \rceil$ there are $\Omega(n)$ many $k$-cycles in $G$. For $k \in \{ 3, 4, 5 \}$ a linear lower bound on the number of $k$-cycles follows from Euler's formula and results of Hakimi and Schmeichel~\cite{HS79}, as mentioned in the Introduction. The case $k = 6$ can be inferred from the first part of this proof. So we may assume $7 \le k \le \lceil (\frac{n-3}{2})^{\log_32} \rceil + 3$. An edge $uv \in E(G)$ is said to be \emph{good} if $u$ and $v$ have precisely two common neighbours $w_1$ and $w_2$, and $u$ or $v$ has degree at most 6. We also call $w_1 u v w_2$ and $w_1 v u w_2$ \emph{zigzag-paths} with the \emph{internal edge} $uv$, and we identify a zigzag-path with its edge set. We define $\mathcal{P}$ to be the set of zigzag-paths having a good internal edge.

Let $v \in V(G)$. Observe that $N_G(v)$ induces an outerplanar subgraph $H_v$ of $G$, and hence $H_v$ has at least two vertices $v_1, v_2$ of degree 2. It is not hard to see that $v v_1$ and $v v_2$ are good edges provided that $v$ has degree at most 6 in $G$. Together with the fact that $G$ has average degree less than 6, we conclude that $|\mathcal{P}| \in \Omega(n)$.

For any $P \in \mathcal{P}$, we define $\mathcal{C}^{k}_P$ to be the family of $k$-cycles $C$ with $E(C) \supset P$ and $\sigma(P) := w_1 v u w_2$ when $P = w_1 u v w_2$. We claim that $\mathfrak{C}_k = (\mathcal{P}, \{\mathcal{C}^{k}_P\}_{P \in \mathcal{P}}, \sigma)$ is a counting base. It suffices to show that $\mathcal{C}^{k}_P \neq \emptyset$ for every $P \in \mathcal{P}$.
Let $P = w_1 u v w_2$ be an element of $\mathcal{P}$. We may assume that $v$ has degree at most 6.

Recall that a \emph{circuit graph} is a pair $(G_0, C_0)$, where $G_0$ is a 2-connected plane graph and $C_0$ is a facial cycle of $G_0$, such that for any 2-cut $S$ of $G_0$, every component of $G_0 - S$ contains a vertex of $C$.

Now, we take $G_0 := G - u - v$. Since $uv$ is a good edge, $u, v$ lie in a face which is bounded by a cycle, denoted by $C_0$, containing $w_1, w_2$. Therefore $(G_0, C_0)$ is a circuit graph. By~\cite[Lemma~2.3]{SY02} there is a $w_1w_2$-path in $G_0$ of length at least $(\frac{n - 3}{2})^{\log_3 2}$, which can be extended to a cycle $C$ in $G$ containing $P$ of length at least $(\frac{n - 3}{2})^{\log_3 2} + 3$. As $P$ is a zigzag-path, without loss of generality, we may assume that there is precisely one face incident with $v$ lying in the exterior of $C$. Therefore, there are at most five faces incident with $v$ lying in the interior of $C$.
It follows from Lemma~\ref{lem:weaklypancyclic}(ii) that there is a $k$-cycle containing $P$ as $k \ge 7 \ge |N_G(v)| + 1$.
This assures that $\mathcal{C}^{k}_P \neq \emptyset$.

As $|\mathcal{P}| \in \Omega(n)$ and $\mathbf{O}_{\mathfrak{C}_k} \le 5$, applying the Counting Base Lemma to the counting base $\mathfrak{C}_k$ yields that there are $\Omega(n)$ $k$-cycles.
\end{proof}

For a given $n$-vertex triangulation $G$, the radius of $G^*$ is at least logarithmic in $n$, but it can be linear: consider double wheels (i.e.\ the join of $\overline{K_2}$ and a cycle), whose duals are prisms.

It follows from a series of papers that planar 4-connected graphs on $n$ vertices contain a $k$-cycle for every $k \in \{ n - 7, \ldots, n \}$; for details, see~\cite{CHW09}. Taking the proof technique of a work by Alahmadi, Aldred, and Thomassen~\cite{AAT20} into account, it is straightforward to infer the presence of an exponential number of $k$-cycles in 5-connected triangulations on $n$ vertices for every $k \in \{ n - 7, \ldots, n \}$. It remains an open question whether for every 5-connected triangulation $G$ there exists a subset $S$ of the cycle spectrum of $G$ such that $|S| \in \omega(1)$ and $G$ contains exponentially many $k$-cycles for every $k \in S$.

\section{Linearly many long cycles}

As we have mentioned in the introduction, there are infinitely many $n$-vertex triangulations with no $k$-cycle for all $k > 9n^{\log_3 2}$, see~\cite{MM63}. So in the family of all triangulations the minimum number of hamiltonian cycles---and indeed the minimum number of long cycles---is simply 0. However, it is a classic problem to investigate how few hamiltonian cycles may occur if we impose the presence of at least one hamiltonian cycle. (We point out that, from a more general perspective, Alahmadi, Aldred, and Thomassen~\cite{AAT20} observed that a 4-connected hamiltonian triangulation on any fixed surface may have at most $O(n^4)$ cycles.) In 1979, Hakimi, Schmeichel, and Thomassen showed that there is an infinite family of triangulations with exactly four hamiltonian cycles~\cite{HST79}, and this is a lower bound for the minimum number of hamiltonian cycles in any hamiltonian triangulation on at least 5 vertices~\cite{KZ88}. This motivates the following result; therein, we will make use of the function $g : \mathbb{N} \rightarrow \mathbb{N}$, $t \mapsto \left\lceil \frac{(t-3)(t-4)}{12} \right\rceil \hspace{-0.5mm}$ and, in contrast to all other parts of this paper, not restrict ourselves to the spherical case.

\begin{theorem}
	Consider an integer $t \ge 4$ with $t = 0, 3, 4, 7$ mod.~$12$, and let $f : \mathbb{N} \rightarrow \mathbb{R}_{\ge 1}$ be a function such that there exists an $n_0$ such that for all $n \ge n_0$ we have $f(n) < \frac{n}{t} - 1$. Then there exists an infinite family ${\cal G}$ of hamiltonian triangulations of genus $g(t)$ such that every $G \in {\cal G}$ with $n := |V(G)|$ contains $O(f(n)^{2t-3})$ $k$-cycles for any $k \in \{ \lceil n - f(n) \rceil, \ldots, n \}$.
\end{theorem}

\begin{proof}
Let $W$ be the join of $\overline{K_2}$ and a cycle $D$, i.e.\ a double wheel. Let $vw \in E(D)$. $vw$ lies in the boundary of exactly two facial triangles, $avw$ and $bvw$. We call $W - vw + ab$ a \emph{flipped double wheel}, and $w$ its \emph{apex}.

By the Map Colour Theorem of Ringel and Youngs~\cite{RY68} we can embed $K_t$ on a surface of genus $g(t)$. Consider $G \cong K_t$; it is well-known that for the given values of $t$, the embedding is a triangulation. We first treat the case when $t = 4$. Consider the infinite family of triangulations given in Figure~\ref{fig:few-long}. It consists of a complete graph on four vertices $a, b, c, d$ with a flipped double wheel inserted into each of its four triangles. We call these triangles as well as the vertices $a,b,c,d$ \emph{original}. We insert the same number $p$ of vertices into each original triangle. We denote this graph by $G_p$. Put $n := |V(G_p)| = 4p + 4$. Henceforth, we assume $p$ to be large; moreover, the proof uses the embedding of $G_p$ given in Figure~\ref{fig:few-long}.

\begin{figure} [!ht]
	\centering
	\includegraphics[height=65mm]{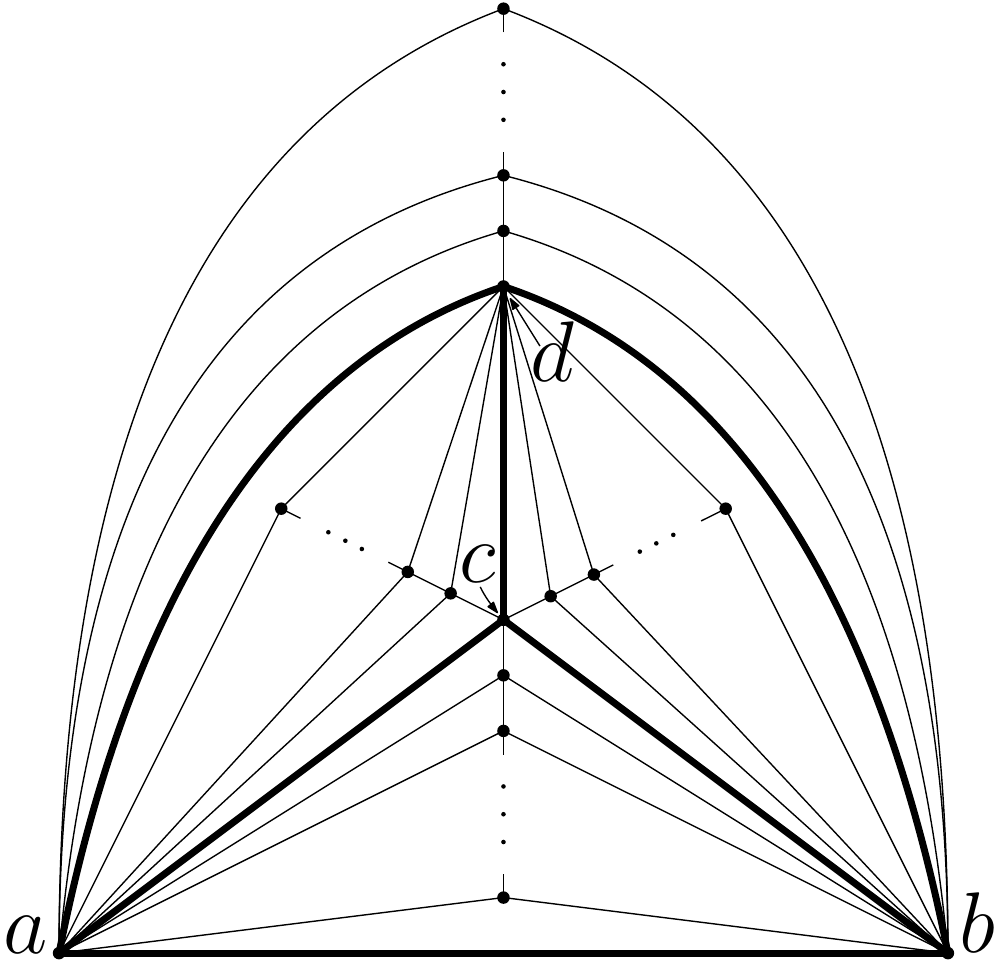}
	\caption{The plane graph $G_p$.
	}
	\label{fig:few-long}
\end{figure}

In $G_p$, consider a $k$-cycle $C$ with $k > 3p + 4$. For any fixed original triangle $\Delta$, the cycle $C$ cannot exit ${\rm int}\Delta$ towards an original vertex and then re-enter ${\rm int}\Delta$. Let $\Delta$ be an original triangle with vertices $x, y, z$ such that $z$ is cubic in ${\rm Int}\Delta$. Consider the path $P := C \cap {\rm Int}\Delta$. Either $P$ is an $xy$-path, in which case we call this traversal of $\Delta$ \emph{flexible}, or $P$ is an $xz$- or a $yz$-path, in which case we call this traversal \emph{fixed}. For any non-negative integer $q \le p$, there are $O(q)$ flexible traversals avoiding exactly $q + 1$ vertices in ${\rm Int}\Delta$, and a constant number of fixed traversals avoiding exactly $q + 1$ vertices in ${\rm Int}\Delta$. As $k > 3p + 4$, the cycle $C$ must visit the vertex $c$ (as defined in Figure~\ref{fig:few-long}) and $ac, bc, cd \notin E(C)$, so at most two traversals of the interiors of original triangles are flexible.

Let $\Delta_1, \Delta_2, \Delta_3, \Delta_4$ be the four original triangles of $G_p$. Put $q := n - k$ and let $q_i$ denote the number of vertices $C$ avoids in ${\rm int}\Delta_i$ so that $q = q_1 + q_2 + q_3 + q_4$. It is well-known that there are $\binom{q+3}{3}$ possible variations for $C$. We have shown above that a $k$-cycle $C$ with $k > 3p + 4$ admits at most two flexible traversals, so that there are in total $O(q^5)$ cycles of length $n - q$ in $G_p$ for all $q < \frac{n}{4} - 1$.

We now use the same strategy for the case when $t > 4$. Consider $v \in V(G)$. Insert into each of the $t - 1$ triangular faces surrounding $v$ a flipped double wheel on $p + 3$ vertices such that its apex is identified with $v$. Among the triangles sharing an edge with $G[N(v)]$ but not incident with $v$, choose one arbitrarily and insert thereinto a flipped double wheel as well. We call the resulting graph $G_p$. In $G_p$, consider a $k$-cycle $C$ with $k > (t-1)p + t$. It is not difficult to see that among the $t$ flipped double wheels present in $G_p$, at least two of $C$'s traversals of flipped double wheels are fixed (as defined in the third paragraph of this proof). This yields at most $t - 2$ flexible traversals. Put $q := n - k$. Arguing as above, there are $\binom{q+t-1}{t-1}$ possibilities for $C$ to avoid $q$ vertices. We obtain in total $O(q^{2t-3})$ cycles of length $n - q$ for all $q < \frac{n}{t} - 1$.
\end{proof}

One can infer immediately from the above result the following corollary.

\begin{corollary}
	For any non-negative integers $c,t$ there exists a constant $c' > 0$ and an infinite family ${\cal G}$ of triangulations of genus $g(t)$ such that any $G \in {\cal G}$ contains at least one and at most $c'$ $k$-cycles for any $k \in \{ \lceil n - c \rceil, \ldots, n \}$. Furthermore, there exists an infinite family ${\cal G}$ of triangulations of the sphere such that any $G \in {\cal G}$ contains at least one and at most $O(n)$ many $k$-cycles for any $k \in \{ \lceil n - \sqrt[5]{n} \rceil, \ldots, n \}$, and an infinite family ${\cal G}'$ of triangulations of the torus such that any $G \in {\cal G}'$ contains at least one and at most $O(n)$ many $k$-cycles for any $k \in \{ \lceil n - \sqrt[11]{n} \rceil, \ldots, n \}$.
\end{corollary}

We now discuss 4-connected triangulations and recall that by a classic theorem of Whitney~\cite{Wh31}, these form a subclass of the family of all hamiltonian triangulations. Alahmadi, Aldred, and Thomassen~\cite{AAT20} showed that every 5-connected triangulation has exponentially many hamiltonian cycles. Lo and Qian~\cite{LQ} extended this result by proving that every 4-connected triangulation with $O(n)$ many 4-cuts has exponentially many hamiltonian cycles. For an integer $n \ge 5$, the join $G$ of $C_{n-2}$ and $\overline{K}_2$---the so-called \emph{double wheel}, which is 4-connected---satisfies $c_3(G) \in \Theta(n)$ and $c_k(G) \in \Theta(n^2)$ for all $k \ge 4$; here, for a given graph $G$, we denote by $c_k(G)$ the number of $k$-cycles occurring in $G$. We emphasise that thus, perhaps contrary to intuition, in 4-connected triangulations no better solution to Problem~~\ref{pro:problem1} than a quadratic one can be achieved. Double wheels have $\Theta(n^2)$ 4-cuts, and no 4-connected triangulation with fewer hamiltonian cycles than a double wheel is known. Indeed, Hakimi, Schmeichel, and Thomassen~\cite{HST79} conjecture that every 4-connected $n$-vertex triangulation $G$ has at least $2(n-2)(n-4)$ hamiltonian cycles\footnote{Very recently, this conjecture has been solved \emph{asymptotically}, see~\cite{LWY22}.}, with equality if and only if $G$ is a double wheel. As already pointed out by Lo and Qian, it is intriguing to observe that few 4-cuts seem to imply many hamiltonian cycles, and vice-versa.

The next theorem's first part answers Problem~\ref*{pro:problem2} affirmatively, while its second part proves that, under certain circumstances, 4-connected triangulations contain a quadratic number of $k$-cycles for many values of $k$. We emphasise here the important open problem to determine whether there exists a $k$ such that every $n$-vertex triangulation contains $\Omega(n^2)$ $k$-cycles.

\begin{theorem} \label{thm:4connected}
	\begin{enumerate}[label=(\roman*)]
		\item Every $n$-vertex triangulation with at most one separating triangle contains $\Omega(n)$ $k$-cycles for every $k \in \{ 3, \ldots, n \}$.
		\item Let $G$ be a $4$-connected $n$-vertex triangulation. Suppose there is a set ${\cal S}$ of separating $4$-cycles in $G$ with
		$$\iota := \big\vert\{ (C, C') \in {\cal S \times \cal S}: C \cap C' \cong K_2 \}\big\vert.$$
		Then $G$ contains $\Omega(|\mathcal{S}| / (\iota + 1))$ $k$-cycles for every $k \in \{ \lceil \frac{n}{2} + \sqrt{\frac{n}{2} - 2} \rceil + 2, \ldots, n \}$.
	\end{enumerate}
\end{theorem}

\begin{proof}
(i) First, assume that $7 \le k \le n$. Let $G$ be a triangulation on $n$ vertices. For any edge $uv$ in $G$, there are exactly two vertices $w_1, w_2$ such that $\{u, v, w_1\}$ and $\{u, v, w_2\}$ induce triangular faces in $G$. We call $w_1 u v w_2$ and $w_1 v u w_2$ \emph{zigzag-paths} with \emph{internal vertices} $u, v$. We may identify a zigzag-path with its edge set, and define $\mathcal{P}$ to be the set of zigzag-paths that have all internal vertices of degree at least 4 and at least one internal vertex of degree at most 6. As $G$ has at most one cubic vertex and average degree less than 6, we have that $|\mathcal{P}| \in \Omega(n)$.
For any $P \in \mathcal{P}$, we define $\mathcal{C}^{k}_P$ to be the family of $k$-cycles $C$ with $E(C) \supset P$ and $\sigma(P) := w_1 v u w_2$ when $P = w_1 u v w_2$.

We claim that $\mathfrak{C}_k = (\mathcal{P}, \{\mathcal{C}^{k}_P\}_{P \in \mathcal{P}}, \sigma)$ is a counting base. It is left to show that every element of $\mathcal{P}$ can be extended to a cycle of length $k$ for any given $k \ge 7$.

Let $k$ be an integer satisfying $7 \le k \le n$. We now show that $\mathcal{C}^{k}_P$ is non-empty for every $P \in \mathcal{P}$. By~\cite[Lemma~14(i)]{BZ19} there is a hamiltonian cycle $\mathfrak{h}$ in $G$ that contains $P$. Denote by $v$ an internal vertex of $P$ of degree at most 6. As $P$ is a zigzag-path, without loss of generality, we may assume that there is precisely one face incident with $v$ lying in the exterior of $\mathfrak{h}$. Therefore, there are at most five faces incident with $v$ lying in the interior of $\mathfrak{h}$.
It follows from Lemma~\ref{lem:weaklypancyclic}(ii) that there is a $k$-cycle containing $P$ in ${\rm Int}\mathfrak{h}$ as $k \ge 7 \ge |N_G(v)| + 1$. Hence we conclude that $\mathcal{C}^{k}_P \neq \emptyset$.

As $|\mathcal{P}| \in \Omega(n)$ and $\mathbf{O}_{\mathfrak{C}_k} \le 5$, we resolve Problem~\ref{pro:problem2} for $k \ge 7$ by applying the Counting Base Lemma to the counting base $\mathfrak{C}_k$. As already mentioned in the proof of Theorem~\ref{thm:linearlyshortcycles}, for $k \in \{ 3, 4, 5 \}$ we get a linear lower bound on the number of $k$-cycles by Euler's formula and results of Hakimi and Schmeichel~\cite{HS79}, while the case $k = 6$ follows from Theorem~\label{thm:linearlyshortcycles}.

\smallskip

\noindent(ii) In this part, we define $\mathcal{P}$ to be the set of pairs of independent edges in 4-cycles from $\mathcal{S}$, i.e.\ for any $P \in \mathcal{P}$, there exists a 4-cycle $v_1 v_2 v_3 v_4 v_1 \in \mathcal{S}$ such that $P = \{v_1 v_2, v_3 v_4\}$. Furthermore, we define $\sigma(P) := \{v_1 v_4, v_2 v_3\}$ and $\mathcal{C}^k_P$ to be the family of $k$-cycles $C$ satisfying that $P \subset E(C)$ and $C - P$ is the union of a $v_1 v_3$-path and a $v_2 v_4$-path. To show that $\mathfrak{C}_k := (\mathcal{P}, \{\mathcal{C}^{k}_P\}_{P \in \mathcal{P}}, \sigma)$ is a counting base, it is left to verify that $\mathcal{C}^k_P$ is non-empty.

Let $P \in \mathcal{P}$ with $P = \{v_1 v_2, v_3 v_4\}$ and $C = v_1 v_2 v_3 v_4 v_1 \in {\cal S}$. Without loss of generality we assume that $|{\rm int} C| \le \frac{n}{2} - 2$. Put $j := |{\rm ext}C| \ge \frac{n}{2} - 2$, $d_1 := d_{{\rm Ext}C - v_2 - v_4}(v_1, v_3)$, and $d_2 := d_{{\rm Ext}C - v_1 - v_3}(v_2, v_4)$.
We claim that $(d_1 - 1) (d_2 - 1) \le j$.
Denote by $L_i$ the set of vertices of ${\rm Ext}C - v_2 - v_4$ at distance $i$ from $v_1$. It is not difficult to see that for any $i \in \{1, \dots, d_1 - 1\}$ there is a path $P^i$ in the graph induced by $L_i$ with one end-vertex in $N(v_2)$ and one other in $N(v_4)$ (here we include the case that $P^i$ consists of only one vertex which is in $N(v_2) \cap N(v_4)$). Note that $P^i$ has length at least $d_2 - 2$. Thus $j \ge \sum_{i = 1}^{d_1 - 1} |L_i| \ge (d_1 - 1) (d_2 - 1)$, from which we obtain that $d_1 \le \sqrt{j} + 1$ or $d_2 \le \sqrt{j} + 1$;
without loss of generality assume the former holds.

By the theorem of Thomas and Yu~\cite{TY94} stating that in a planar 4-connected graph the removal of any pair of vertices yields a hamiltonian graph, ${\rm Ext}C - v_2 - v_4$ contains a hamiltonian $v_1 v_3$-path.
By applying Lemma~\ref{lem:weaklypancyclic}(i) to a $(d_1 + 2)$-cycle containing $v_1 v_2 v_3$ and to a $(j + 3)$-cycle containing $v_1 v_2 v_3$ in ${\rm Ext}C - v_4$, respectively, we obtain that ${\rm Ext}C - v_2 - v_4$ contains a $v_1 v_3$-path of length $k$ for every integer $k$ satisfying $\sqrt{j} + 1 \le k \le j + 1$. Once more invoking the aforementioned theorem of Thomas and Yu, we know that there exists a hamiltonian $v_2 v_4$-path in ${\rm Int}C - v_1 - v_3$. As a cycle can be obtained by adding edges $v_1 v_2, v_3 v_4$ to the union of a $v_1 v_3$-path in ${\rm Ext}C - v_2 - v_4$ and a $v_2 v_4$-path in ${\rm Int}C - v_1 - v_3$, for every integer $k$ satisfying $\frac{n}{2} + \sqrt{\frac{n}{2} - 2} + 2 \le k \le n$ we have that $\mathcal{C}^{k}_P \ne \emptyset$ for every $P \in \mathcal{P}$. We can conclude that for these $k$, the triple $\mathfrak{C}_k = (\mathcal{P}, \{\mathcal{C}^{k}_P\}_{P \in \mathcal{P}}, \sigma)$ is a counting base.

Note that $|\mathcal{P}| = 2 |\mathcal{S}|$ and $\mathbf{O}_{\mathfrak{C}_k} \le \iota + 1$. We complete the proof by applying the Counting Base Lemma to the counting base $\mathfrak{C}_k$.
\end{proof}

If there are two vertices having $t$ common neighbours, then there is a set ${\cal S}$ of separating $4$-cycles with $|\mathcal{S}| \in \Omega(t^2)$ and
$\iota := \vert\{ (C, C') \in {\cal S \times \cal S}: C \cap C' \cong K_2 \} \vert = 0$. This yields the following corollary of Theorem~\ref{thm:4connected}(ii).

\begin{corollary} \label{cor:n2}
	Every $4$-connected $n$-vertex triangulation with two vertices having $\Theta(n)$ common neighbours contains $\Omega(n^2)$ $k$-cycles for every $k \in \{ \lceil \frac{n}{2} + \sqrt{\frac{n}{2} - 2} \rceil + 2, \ldots, n \}$.
\end{corollary}

For every integer $k \ge 4$, Theorem~\ref{thm:4connected}(i) is not true for triangulations with $k$ separating triangles as there exist infinite families of such triangulations with a constant number of hamiltonian cycles; for $k \in \{ 4, 5 \}$ see~\cite{BSV18}, and for $k \ge 6$, infinitely many non-hamiltonian examples are easy to describe using toughness arguments. We do not know whether Theorem~\ref{thm:4connected}(i) can be extended to all triangulations containing at most two or at most three separating triangles.

We conclude this section with an application of Theorem~\ref{thm:4connected}(i) to general (i.e.\ not necessarily 4-connected) triangulations.

\begin{corollary}
	Let $G$ be an $n$-vertex triangulation containing a triangle $\Delta$ such that ${\rm Int}\Delta$ is a triangulation containing at most one separating triangle and $\Theta(n)$ vertices. Then $G$ contains $\Omega(n)$ many $k$-cycles for all $k \in \{ 3, \ldots, |V({\rm Int}\Delta)| \}$.
\end{corollary}

\section{Notes}

\noindent \textbf{1.} It remains an open question whether there exists a $k$ such that every $n$-vertex triangulation contains $\Omega(n^2)$ many $k$-cycles. We do know by Theorem~\ref{thm:4connected}(ii) and Corollary~\ref{cor:n2} that under certain conditions 4-connected $n$-vertex triangulations contain $\Omega(n^2)$ $k$-cycles for many values of $k$, and that \emph{every} such triangulation contains $\Omega(n^2)$ many $(n - 2)$-cycles by a theorem of Thomas and Yu~\cite{TY94}, and $\Omega(n^2)$ many $n$-cycles by a very recent breakthrough result of Liu, Wang, and Yu~\cite{LWY22}. Although not explicitly stated, the latter work also implies the presence of $\Omega(n^2)$ many $(n - 1)$-cycles in 4-connected $n$-vertex triangulations.

\smallskip

\noindent \textbf{2.} As mentioned in the introduction, Mohar and Shantanam~\cite{MS} showed that in an 4-connected $n$-vertex triangulation $G$, every edge in $G$ is contained in a $k$-cycle for every $k \in \{ 3, \ldots, n \}$. It would be interesting to investigate to which degree this can be strengthened in an enumerative sense. While there are always exactly two 3-cycles through any given edge, already for 4-cycles the situation is more diverse: any edge in any triangulation on at least five vertices lies on at least four 4-cycles, but there exist triangulations such as the double wheels in which many edges lie on a linear (in the graph's order) number of 4-cycles. Furthermore, there exists a length $k$, depending on $n$, for which we can guarantee that every edge in every 4-connected triangulation lies in an at least linear number of $k$-cycles: it follows from a result of Ozeki and Vr\'ana~\cite{OV14} that, given a planar 4-connected graph $G$, for any 2-element set $X \subset E(G) \cup E(\overline{G})$, the graph $G \cup X$ admits a hamiltonian cycle containing $X$. Thus, for any pairwise distinct $v,w_1,w_2 \in V(G)$, where $w_1$ and $w_2$ are adjacent, we obtain a hamiltonian cycle in $G - v$ passing through $w_1w_2$. Hence, for any $n$-vertex planar 4-connected graph $G$ and any edge $e$ therein, $G$ contains $\Omega(n)$ many $(n-1)$-cycles containing $e$. In fact, double wheels show that one cannot do better than a linear bound for $(n-1)$-cycles through a specific edge.

\smallskip

\noindent \textbf{3.} In analogy to Tutte's celebrated result that planar 4-connected graphs are hamiltonian~\cite{Tu56}, one can ask how the problems and conjectures in this paper relate to the larger family of planar 4-connected graphs. While this class is known to contain $\Omega(n)$ hamiltonian cycles~\cite{BV21}, in stark contrast to 4-connected triangulations, some of its members are not even pancyclic, as 4-cycles might not be present (consider the line-graph of the dodecahedron).

\smallskip

\bigskip

\noindent \textbf{Acknowledgements.} On-Hei Solomon Lo's work was partially supported by Natural Sciences and Engineering Research Council of Canada (NSERC). Carol T.~Zamfirescu's research was supported by a Postdoctoral Fellowship of the Research Foundation Flanders (FWO).


\begin{thebibliography}{99}

\bibitem{AAT20}
A. Alahmadi, R. E. L. Aldred, and C. Thomassen.
Cycles in 5-connected triangulations.
\emph{J. Combin. Theory, Ser. B} \textbf{140} (2020) 27--44.

\bibitem{Al80}
A. F. Alameddine.
On the number of cycles of length 4 in a maximal planar graph
\emph{J. Graph Theory} \textbf{4} (1980) 417--422.

\bibitem{BSV18}
G. Brinkmann, J. Souffriau, and N. Van Cleemput.
On the number of hamiltonian cycles in triangulations with few separating triangles.
\emph{J. Graph Theory} \textbf{87} (2018) 164--175.

\bibitem{BV21}
G. Brinkmann and N. Van Cleemput.
4-connected polyhedra have at least a linear number of hamiltonian cycles.
\emph{Europ. J. Combin.} \textbf{97} (2021) Article number 103395.

\bibitem{BZ19}
G. Brinkmann and C. T. Zamfirescu.
Polyhedra with few 3-cuts are hamiltonian.
\emph{Electron. J. Combin.} \textbf{26} (2019) \#P39.

\bibitem{CY02}
G. Chen and X. Yu.
Long cycles in 3-connected graphs.
\emph{J. Combin. Theory, Ser. B} \textbf{86} (2002) 80--99.

\bibitem{CHW09}
Q. Cui, Y. Hu, and J. Wang.
Long cycles in 4-connected planar graphs.
\emph{Discrete Math.} \textbf{309} (2009) 1051--1059.

\bibitem{GPSTZ}
E. Gy\H{o}ri, A. Paulos, N. Salia, C. Tompkins, and O. Zamora.
The Maximum Number of Pentagons in a Planar Graph.
arXiv:1909.13532 [math.CO].

\bibitem{HS79}
S. L. Hakimi and E. F. Schmeichel.
On the Number of Cycles of Length $k$ in a Maximal Planar Graph.
\emph{J. Graph Theory} \textbf{3} (1979) 69--86.

\bibitem{HST79}
S. L. Hakimi, E. F. Schmeichel, and C. Thomassen.
On the Number of Hamiltonian Cycles in a Maximal Planar Graph.
\emph{J. Graph Theory} \textbf{3} (1979) 365--370.

\bibitem{HJW}
T. Huynh, G. Joret, and D. R. Wood.
Subgraph densities in a surface.
arXiv:2003.13777 [math.CO]. Published online in \emph{Combinatorics, Probabiliy, and Computing}, see
https://doi.org/10.1017/S0963548321000560.

\bibitem{KZ88}
J. Kratochvil and D. Zeps.
On the number of Hamiltonian cycles in triangulations.
\emph{J. Graph Theory} \textbf{12} (1988) 191--194.

\bibitem{LY21}
X. Liu and X. Yu.
Number of Hamiltonian Cycles in Planar Triangulations.
\emph{SIAM J. Discrete Math.} \textbf{35} (2021) 1005--1021.

\bibitem{LWY22}
X. Liu, Z. Wang, and X. Yu.
Counting Hamiltonian cycles in planar triangulations.
\emph{J. Combin. Theory, Ser. B} \textbf{155} (2022) 256--277.

\bibitem{Lo20}
O.-H. S. Lo.
Hamiltonian cycles in 4-connected plane triangulations with few 4-separators.
\emph{Discrete Math.} \textbf{343} (2020) Article Nr.~112126.

\bibitem{Lo22}
O.-H. S. Lo.
Distribution of subtree sums.
To appear in \emph{Europ. J. Combin.}

\bibitem{LQ}
O.-H. S. Lo and J. Qian.
Hamiltonian cycles in 4-connected planar and projective planar triangulations with few 4-separators.
\emph{SIAM J. Discrete Math.} \textbf{36} (2022) 1496--1501.

\bibitem{MS}
B. Mohar and A. Shantanam.
Pancyclicity in 4-connected planar graphs.
Manuscript.

\bibitem{MM63}
J. W. Moon and L. Moser.
Simple paths on polyhedra.
\emph{Pacific J. Math.} \textbf{13} (1963) 629--631.

\bibitem{OV14}
K. Ozeki and P. Vr\'ana.
2-edge-Hamiltonian-connectedness of 4-connected plane graphs.
\emph{Europ. J. Combin.} \textbf{35} (2014) 432--448.

\bibitem{RY68}
G. Ringel and J. W. T. Youngs.
Solution of the Heawood map-coloring problem.
\emph{Proc. Nat. Acad. Sci.} \textbf{60} (1968) 438--445.

\bibitem{Sa96}
D. P. Sanders.
On Hamiltonian cycles in certain planar graphs.
\emph{J. Graph Theory} \textbf{21} (1996) 43--50.

\bibitem{SY02}
L. Sheppardson and X. Yu.
Long cycles in 3-connected graphs in orientable surfaces.
\emph{J. Graph Theory} \textbf{41} (2002) 69--84.

\bibitem{TY94}
R. Thomas and X. Yu.
4-Connected Projective-Planar Graphs are Hamiltonian.
\emph{J. Combin. Theory, Ser. B} \textbf{62} (1994) 114--132.

\bibitem{Tu56}
W. T. Tutte.
A theorem on planar graphs.
\emph{Trans. Amer. Math. Soc.} \textbf{82} (1956) 99--116.

\bibitem{Wh31}
H. Whitney.
A theorem on graphs.
\emph{Ann. Math.} \textbf{32} (1931) 378--390.

\end{thebibliography}
\end{document}